\documentclass{amsart}
\usepackage{amsfonts,amscd}

\newtheorem{theorem}{Theorem}[section]
\newtheorem{lemma}[theorem]{Lemma}
\newtheorem{corollary}[theorem]{Corollary}
\newtheorem{proposition}[theorem]{Proposition}
\theoremstyle{remark}

\theoremstyle{definition}

\numberwithin{equation}{section}
\makeatother

\DeclareMathOperator{\Cdb}{{\mathbb C}}

\DeclareMathOperator{\Ndb}{{\mathbb N}}

\begin{document}

\title[Operator algebras: Weak compactness and spectrum]{Operator algebras with contractive approximate identities: Weak compactness and the spectrum}

\thanks{The first author was supported by a grant DMS  1201506 from the NSF.
The second author is grateful for  support from UK research council
grant  EP/K019546/1}

\author{David P. Blecher}
\address{Department of Mathematics, University of Houston, Houston, TX
77204-3008}
\email[David P. Blecher]{dblecher@math.uh.edu}

\author{Charles John Read}
\address{Department of Pure Mathematics,
University of Leeds,
Leeds LS2 9JT,
England}
 \email[Charles John Read]{read@maths.leeds.ac.uk}

\begin{abstract}  We continue our study of operator algebras with contractive approximate identities
(cais) by presenting a couple of interesting examples of operator algebras with cais,
which in particular answer questions raised in previous papers in this series, for example
about whether, roughly speaking, 
`weak compactness' of an operator algebra, or the lack of it, can be seen in the spectra
of its elements. \end{abstract}

\maketitle

\section{Introduction}

An {\em operator algebra} is a closed subalgebra of $B(H)$, for a
Hilbert space $H$.   An operator algebra with a
 contractive approximate identity (cai) is called
 {\em approximately unital}.
Here we construct an interesting new approximately unital operator algebra, and use it to solve questions arising in our earlier work, for example
about whether, roughly speaking, 
`weak compactness' of an operator algebra, or the lack of it, can be seen in the spectra
of its elements.  We now describe some background for this.   We recall that a semisimple Banach algebra $A$ is
a modular annihilator algebra iff no element of $A$ has a nonzero
limit point in its spectrum \cite[Theorem 8.6.4]{Pal}.
 If $A$ is also commutative then this is equivalent to
the Gelfand spectrum of $A$ being discrete \cite[p.\ 400]{LN}.
We write $M_{a,b} : A \to A : x \mapsto axb$,
where $a, b \in A$.   Recall that a Banach algebra
is {\em compact} if the map $M_{a,a}$ is compact
for all $a \in A$.  We say that $A$ is {\em weakly compact} if
$M_{a,a}$ is weakly compact for all $a \in A$.  If $A$ is approximately unital  and commutative then $A$ is
weakly compact iff $A$ is an  ideal  in its bidual  $A^{**}$
(see e.g.\ \cite[1.4.13]{Dal}).
 In the noncommutative case  $A$ is
weakly compact  iff $A$ is a
{\em hereditary subalgebra} (or {\em HSA}) in its bidual 
(see \cite[Lemma 5.1]{ABR}).
It is known  \cite{Pal} that every compact semisimple Banach algebra
is a modular annihilator algebra (and conversely every semisimple `annihilator algebra', or
more generally any Banach algebra
with dense socle, is compact).  Thus it is of interest to know
if there are any connections for operator algebras
between being a semisimple modular annihilator algebra, and being
weakly compact.  See the discussion after Proposition 5.6 in \cite{ABR}, where some specific questions
along these lines are raised.  We have solved these here; indeed
we have by now solved essentially all open 
questions posed in our previous papers
\cite{BRI, BRII, ABR}.  In particular we show here, first, that 
 a semisimple approximately unital
operator algebra which is a modular annihilator algebra
need not be weakly compact, nor need it be nc-discrete.  
(The latter term  will be defined before Corollary \ref{iscr}, 
when it is needed.)   
Second,  an approximately unital 
commutative weakly compact
semisimple operator algebra $A$ need not have countable or scattered spectrum
(in fact the spectrum of some of its elements can have nonempty interior).

\section{A semisimple operator algebra which is a modular annihilator algebra but is not weakly compact}
  \label{notw}

Let $(c_n)$ be an unbounded increasing sequence in $(0,\infty)$.  For each $n \in \Ndb$ let
$d_n$ be the diagonal matrix in $M_n$ with
$c_n^k$ as the $k$th  diagonal entry.  If $M$ is the von Neumann algebra $\oplus_n^\infty \, (M_n  \oplus
M_n)$, we let $N$ be its weak*-closed unital subalgebra consisting of tuples $((x_n , d_n x_n d_n^{-1}))$,
for all $(x_n) \in \oplus_n^\infty \, M_n$.       We define $A_{00}$ to
be the finitely supported tuples in $N$, and  $A_0$ to be the closure of $A_{00}$.
 That is, $A_0$ is the intersection of the
$c_0$-sum $C^*$-algebra  $\oplus_n^{\circ} \, (M_n  \oplus
M_n)$ with $N$.   We sometimes simply write $(x_n)$ for the associated tuple in $N$.

\begin{lemma} \label{lem1}  Let $A$ be any closed subalgebra of $N$ containing $A_0$.  Then $A$ is
semisimple.  \end{lemma} \begin{proof}  For any nonzero $x = (x_n) \in A$, choose
$m$ and $i$ with $z = x_m e_i \neq 0$, where $(e_i)$ is the usual basis of
$\Cdb^m$.   Choose $y_m \in M_m$ with $y_m z = e_i$, and otherwise set $y_n = 0$.
Then $y  = (y_n) \in A_0$, and the copy of $e_i$ is in the kernel of $I-yx$.
Hence $I - yx$ is not invertible
in $A^1$, and so $x$ is not in the Jacobson radical by a well known
characterization of that radical.  Thus $A$ is
semisimple.  \end{proof}

Endow $M_n$ with a norm $p_n(x) = \max \{ \Vert x \Vert , \Vert d_n x d_n^{-1} \Vert \}$.
Then $N \cong \oplus^\infty_n \, (M_n, p_n(\cdot))$ isometrically, and we
write $p(\cdot)$ for the norm on the latter space, so
$p((x_n)) = \sup_n \, p_n(x_n)$.   We sometimes view $p$ as the  norm on $N$ via
the above identification.
Let $L_n$ be the left shift on $\Cdb^n$, so that in particular $L_n e_1 = 0$.
Note that $d_n L_n d_n^{-1} = \frac{1}{c_n} \, L_n$, and that $p_n(L_n) = 1$ if $n \geq 2$.
For $n, k \in \Ndb$ with $n \geq k$ define an `integer interval' $E_{n,k} = \Ndb_0 \cap [\frac{n}{k} ,
\frac{2n}{k}]$.  Set $\mu_{n,k} = |E_{n,k}|$ if $n \geq k$, with $\mu_{n,k} = 1$ if
$n < k$.  Then $\mu_{n,k}$ is strictly positive for all $n, k$.  For $n \geq k$
define $u_{n,k} = \frac{1}{\mu_{n,k}} \, \sum_{i \in E_{n,k}} \, (L_n)^i \in M_n$.
If $n < k$ set $u_{n,k} = I_n$.   Define $u_k = (u_{n,k})_{n \in \Ndb}$.  We have
$$p_n(u_{n,k}) \leq \max_{i \in E_{n,k}} \, p((L_n)^i) \leq 1 , \qquad n \geq k,$$
and so  $$p(u_k) \leq 1 , \qquad k \in \Ndb .$$

The operator algebra we are interested in is
$$A = \{ a \in N : p(a u_k -a) + p(u_k a - a) \to 0 \} .$$
This will turn out to be
the largest subalgebra of $N$ having $(u_k)$ as a cai.
First,   a preliminary estimate:

\begin{lemma} \label{lem2}  Let $L \in {\rm Ball}(B)$ for a Banach algebra $B$.
Suppose that  $E_1$ is a set of $\mu_1$ integers from $[0, n]$, and
$E_2$ is a set of $\mu_2$ consecutive nonnegative integers.    If
$u_i = \frac{1}{\mu_{i}} \, \sum_{i \in E_{i}} \, L^i$ then
$$\Vert u_1 u_2 - u_2 \Vert \leq \frac{2 n}{\mu_2}.$$
 \end{lemma} \begin{proof}   If $n \geq \mu_2$
then
$$\Vert u_1 u_2 - u_2 \Vert \leq \Vert u_1 \Vert \Vert u_2  \Vert + \Vert u_2  \Vert
\leq 2 \leq \frac{2n}{\mu_2}.$$
So we may assume that $n < \mu_2$.  Let $m_0 = \min \, E_2$.
Then
$$u_1 u_2 = \frac{1}{\mu_1 \, \mu_2} \, \sum_{j \in E_1, k \in E_2} \,
 L^{j+k}
= \sum_{m_0 \leq m < m_0 + n + \mu_2} \, \lambda_m L^m , $$
where $\lambda_m$ is $\frac{1}{\mu_1 \, \mu_2}$ times the number of pairs in
$E_1 \times E_2$ which sum to $m$.
  Since
$$\mu_1 \leq n+ 1 \leq \mu_2,$$
and since the number of such pairs cannot exceed $\mu_1 = |E_1|$,
we have $$0 \leq \lambda_m \leq \frac{1}{\mu_2}.$$
If $m \in [m_0 + n, m_0 + \mu_2)$ then $m - k \in E_2$
for any integer $k$ in $[0,n]$, and so $m - E_1 \subset E_2$.
We deduce that $$\lambda_m = \frac{1}{\mu_2} , \qquad m \in [m_0 + n, m_0 + \mu_2
).$$ Since $u_2 = \frac{1}{\mu_2}  \, \sum_{m_0 \leq m < m_0 + \mu_2} \, L^m$ we have
$$u_1 u_2 - u_2 = \sum_{m_0 \leq m < m_0 + n} \,
(\lambda_m - \frac{1}{\mu_2}) \, L^m \, + \, \sum_{m_0 + \mu_2
\leq m <  m_0 + n + \mu_2} \, \lambda_m L^m .$$
No coefficient in the last sum has modulus greater than $\frac{1}{\mu_2}$, and there
are $2n$ nonzero coefficients, so
$$\Vert u_1 u_2 - u_2 \Vert \leq \frac{2n}{\mu_2} \, \max_m \,
\Vert L^m \Vert = \frac{2n}{\mu_2} $$
as desired.   \end{proof}

\begin{corollary} \label{isab}  Let $A = \{ a \in N : p(a u_k -a) + p(u_k a - a) \to 0 \}$.
Then   $A$ is a semisimple operator algebra with cai $(u_k)$, and $A_0$
is an ideal in $A$.
\end{corollary}
\begin{proof}   We first show $u_r \in A$ for all $r \in \Ndb$.  Let $k \geq r$.
If $n \geq k$
then $E_{n,k}$ is a subset of $[0, \frac{2n}{k}]$, and $\mu_{n,k}$ is
either $\lfloor \frac{n}{k} \rfloor$ or $\lfloor \frac{n}{k}+1 \rfloor$.
By Lemma \ref{lem2}, we have
$$p_n(u_{n,k} u_{n,r} - u_{n,r}) \leq \frac{2 \, \lfloor \frac{2n}{k} \rfloor}{\lfloor \frac{n}{r} \rfloor} ,
r \geq n \geq k .$$
If $n < k$ then $p_n(u_{n,k} u_{n,r} - u_{n,r}) = 0$.  If $k \geq 2 t r$ for an integer
$t > 1$ then $$\frac{2 \, \lfloor \frac{2n}{k} \rfloor}{\lfloor \frac{n}{r} \rfloor} \leq
\frac{\lfloor\frac{n}{tr} \rfloor}{\lfloor \frac{n}{r} \rfloor} \leq \frac{1}{t} .$$
Thus $p_n(u_{n,k} u_{n,r} - u_{n,r}) \leq \frac{2}{t}$ for $k \geq 2 t r$, so
$$p(u_{k} u_r - u_r) = \sup_n \, p_n(u_{n,k} u_{n,r} - u_{n,r}) \leq \frac{2}{t}
, \qquad k \geq 2 t r .$$
So $u_k u_r \to u_r$ with $k$, and so $u_r \in A$ for all $r  \in \Ndb$.

It is now obvious that $A$, being a subalgebra of the operator algebra $N$,
is an operator algebra with cai $(u_k)$.
It is elementary that for any matrix $x$ in
the copy $M_n'$ of $M_n$ in $A_0$ we have $x u_k \to x$ and $u_k x \to x$,
since for example $u_k x = x$ for $k > n$.
 Hence $A_0 \subset A$, so that  $A$ is semisimple by Lemma \ref{lem1}.
Since $M_n'$ is an ideal in $N$, so is $A_0$, giving the last statement.
  \end{proof}

In the following result, and elsewhere, $\Vert \cdot \Vert$ denotes
the usual norm on $M_n$ or on $\oplus^\infty_n \, M_n$.
\begin{lemma} \label{lem3}  For each $n \in \Ndb$ and $k \leq n$,
we have  $\Vert u_{n,k} \Vert \geq 1 - \frac{2}{k}$ and
$\Vert u_{n,k}^3 \Vert \geq 1 - \frac{6}{k}$.
\end{lemma} \begin{proof}
 If $\eta$ is the unit vector $(\frac{1}{\sqrt{n}}, \cdots , \frac{1}{\sqrt{n}})$ in $\Cdb^n$, then it is easy to see that
$$\langle (L_n)^k \eta, \eta
\rangle = 1 - \frac{k}{n} , \qquad 0 \leq k \leq n. $$
Since $u_{n,k}$ is an average of powers $(L_n)^j$ with
$0 \leq j \leq \frac{2n}{k}$, we have
$$\langle u_{n,k} \eta, \eta
\rangle \geq 1 - \frac{\frac{2n}{k}}{n} = 1 - \frac{2}{k} .$$
Similarly, $u_{n,k}^3$ is a weighted average of powers $(L_n)^j$ with
$0 \leq j \leq \frac{6n}{k}$.
 \end{proof}

We note that the diagonal matrix units $e^n_{i,i}$ are orthogonal projections, and are
also minimal idempotents in $A$ (that is,
have the property that
$e A e = \Cdb e$).

\begin{theorem} \label{nowk}  $A$ is not weakly compact, and is not separable.
\end{theorem} \begin{proof}  Note that $A$ is an $\ell^\infty$-bimodule
via the action $$(\alpha_n) \cdot (T_n) =  (T_n)  \cdot (\alpha_n)
= (\alpha_n T_n) , \qquad (\alpha_n) \in \ell^\infty,
(T_n) \in A .$$
We will use this to embed $\ell^\infty$ isomorphically in $xAx$, where
$x = u_{r}$ for large enough $r$.   Note that
$$\ell^\infty \cdot x^3 = x (\ell^\infty \cdot x)x \subset xAx .$$
Choosing $r$ with
$1 - \frac{6n}{r} \geq \frac{1}{2}$, we have  that
$\Vert u_{n,r}^3 \Vert \geq \frac{1}{2}$ for all $n \in \Ndb$
(recall $u_{n,r} = I$ if $n < r$).  Thus  for $\vec \alpha = (\alpha_n) \in \ell^\infty$ we have
 $$p(\vec \alpha \cdot x^3) \geq \Vert \vec \alpha \cdot x^3 \Vert =  \Vert \vec \alpha \cdot u_{r}^3 \Vert = \sup_n \, |\alpha_n| \Vert u_{n,r}^3 \Vert \geq \frac{1}{2} \sup_n \, |\alpha_n|,$$
and so the map $\vec \alpha  \mapsto \vec \alpha \cdot x^3$ is a bicontinuous injection of
$\ell^\infty$  into $xAx$.  Thus  $A$ is not weakly compact, nor separable.
\end{proof}

\begin{lemma} \label{lem4}  If $T = (T_n) \in A$, then $\Vert d_n T_n d_n^{-1} \Vert \to 0$ as $n \to \infty$.
Thus the spectral radius $r(T_n) \to 0$  as $n \to \infty$.
\end{lemma} \begin{proof}
Given $\epsilon > 0$ there exists an  $m \in \Ndb$ such that
$$p_n(u_{n,m} T_n - T_n) + p_n(T _n u_{n,m}  - T_n)  < \frac{\epsilon}{2} p(u_m T - T) + p(T u_m  - T)  < \frac{\epsilon}{2} ,
n \in \Ndb. $$ We have noted that $d_n L_n d_n^{-1} = \frac{1}{c_n} L_n,$ and for $n \geq m$ the operator
$u_{n,m}$ is an average of powers $L^j_n$, so for $n \geq m$ we have
$$\Vert d_n u_{n,m} d_n^{-1} \Vert \leq \max_{j \in \Ndb} \, \Vert d_n L_n^j  d_n^{-1} \Vert \leq  \frac{1}{c_n} .$$
Thus $$\Vert d_n T_n u_{n,m}  d_n^{-1} \Vert \leq \frac{1}{c_n}  \Vert d_n T_n d_n^{-1} \Vert \leq \frac{1}{c_n}  p(T) .$$
Consequently,  for $n \geq m$ the  quantity $\Vert d_n T_n d_n^{-1} \Vert$ is dominated by
$$\Vert d_n (T_n u_{n,m} - T_n) d_n^{-1} \Vert + \Vert d_n T_n u_{n,m} d_n^{-1} \Vert \leq
p_n(T_n u_{n,m} - T_n)  + \frac{1}{c_n}  p(T) \leq \frac{\epsilon}{2}  + \frac{1}{c_n}  p(T).$$
The result is clear from this.
\end{proof}

For a matrix $B$ write $\overline{\Delta}_U B$ for the upper triangular projection of $B$ (that is, we change $b_{ij}$ to $0$ if
$i > j$).  Similarly, write $\Delta_L B$ for the strictly lower triangular part of $B$.
In the next results, as usual ${r \choose s} = 0$ if $0 \leq r < s$ are integers.

\begin{lemma} \label{lem5}  If $0 \neq T = (T_n) \in A$, and $\epsilon > 0$ is given, there exist
$k, m \in \Ndb$ such that for all $r \in \Ndb_0$ and $n \geq \max \{ k, m \}$, we have
$$\Vert (\overline{\Delta}_U T_n)^r \Vert \leq \sum_{s=0}^{k-1} \, {r \choose s} \, (2 p(T))^r \, \epsilon^{r-s} .$$
\end{lemma} \begin{proof}
The $i$-$j$ entry $T_{n,i,j}$ of $T_n$ equals
$\langle T_n e_j , e_i \rangle = c_n^{j-i} \, \langle  d_n T_n d_n^{-1} e_j , e_i \rangle$, and so
$$|T_{n,i,j} | = c_n^{j-i} \, |\langle  d_n T_n d_n^{-1} e_j , e_i \rangle|  \leq c_n^{j-i} \, p_n(T_n) , \qquad  T = (T_n) \in A.$$
It follows from this that
$$\Vert \sum_{j = 1}^{n-r} \, T_{n,j+r,j} \, E_{j+r,j} \Vert = \max_{j \leq n-r} \, |T_{n,j+r,j} | \leq c_n^{-r} \, p_n(T_n) ,$$ if $r < n$.
Since $\sum_{r =1}^{n-1} \, (\sum_{j = 1}^{n-r} \, T_{n,j+r,j} \, E_{j+r,j}) = \Delta_L T_n$, we deduce that
\begin{equation} \label{lse} \Vert \Delta_L T_n \Vert  = \Vert T_n -  \overline{\Delta}_U T_n 
\Vert \leq \sum_{r =1}^{n-1} \, c_n^{-r} \, p_n(T_n) \leq \frac{p_n(T_n)}{c_n - 1} \leq
\frac{p(T)}{c_n - 1}. \end{equation}

Given $\epsilon > 0$ choose $k$ with $p(u_k T - T) < \epsilon p(T),$  and let $n \geq k$.
 Then $$\Vert u_{n,k} T_n - T_n \Vert \leq p_n(u_{n,k} T_n - T_n) < \epsilon p(T),$$ and so
$$\Vert u_{n,k} \overline{\Delta}_U T_n - \overline{\Delta}_U T_n \Vert  \leq \epsilon p(T) + \Vert (u_{n,k} - I) (T_n -  \overline{\Delta}_U T_n) \Vert \leq p(T) (\epsilon + \frac{2}{c_n - 1}),$$
since $$u_{n,k} \overline{\Delta}_U T_n - \overline{\Delta}_U T_n = (I - u_{n,k}) (T_n -  \overline{\Delta}_U T_n)  + (u_{n,k} T_n - T_n) .$$
Let $S_1 = u_{n,k} \overline{\Delta}_U T_n$ and $S_2 = \overline{\Delta}_U T_n - S_1$, then
$\Vert S_2 \Vert \leq p(T) (\epsilon + \frac{2}{c_n - 1}),$ by the last displayed equation.   Also,
$$\Vert S_1 \Vert \leq \Vert \overline{\Delta}_U T_n \Vert \leq p(T) + \Vert (I - \overline{\Delta}_U) T_n \Vert \leq p(T) + \frac{p(T)}{c_n - 1}
= p(T) \frac{c_n}{c_n - 1}$$
by (\ref{lse}).

Now $\overline{\Delta}_U T_n = S_1 + S_2$,
so $(\overline{\Delta}_U T_n)^r$ is a sum from $s = 0$ to $r$, of ${r \choose s}$ times terms which are a product of $r$ factors,
$s$ of which are $S_1$ and $r-s$ of which are $S_2$.    Note that any product of upper triangular $n \times n$ matrices that has $k$ or more
factors which equal $S_1$, is zero.  This is because multiplication of an upper triangular matrix $U$ by $u_{n,k}$ (and hence by $S_1$)
decreases the
number of nonzero `superdiagonals' of $B$ by a number
$\geq \frac{n}{k}$, so after $k$ such multiplications we are left with the zero matrix.   Thus we can assume that $s < k$ above.
Using the estimates at the end of the last paragraph, we deduce that
$$\Vert (S_1 + S_2)^r \Vert \leq \sum_{s=0}^{k-1} \, {r \choose s} \Vert S_1 \Vert^s \Vert S_2 \Vert^{r-s} \leq
\sum_{s=0}^{k-1} \, {r \choose s} \, (p(T) \frac{c_n}{c_n - 1})^s \, (p(T) (\epsilon + \frac{2}{c_n - 1}))^{r-s}. $$
Since $c_n \to \infty$ we may choose $m$ such that $\frac{c_n}{c_n - 1} < 2$ and $\epsilon + \frac{2}{c_n - 1} < 2 \epsilon$ for
all $n \geq m$.  Thus for $n \geq \max \{ k, m \}$, we have
$$\Vert (\overline{\Delta}_U T_n)^r  \Vert = \Vert (S_1 + S_2)^r \Vert \leq \sum_{s=0}^{k-1} \, {r \choose s} \,  (2 p(T))^r \epsilon^{r-s}$$
as desired.
\end{proof}

For $k \in \Ndb$ and
positive numbers $b, \epsilon$,
define a quantity  $K(k,b,\epsilon) = \frac{1}{2 b
(1- \epsilon) \,  \epsilon^k}$.

\begin{lemma} \label{lem6}  If $0 \neq T = (T_n) \in A$, and $\epsilon > 0$ is given, there exist
$k, m \in \Ndb$ such that for all $\lambda \in \Cdb$ with $|\lambda| > 4 p(T) \epsilon$,
and $n \geq \max \{ k, m \}$, we have $\lambda I - \overline{\Delta}_U T_n$ and $\lambda I - T_n$ invertible in $M_n$, and both
$$\Vert (\lambda I - \overline{\Delta}_U T_n)^{-1} \Vert \leq K(k,
p(T), \epsilon)$$ and $$\Vert  (\lambda I - T_n)^{-1} \Vert \leq 2 K(k,
p(T), \epsilon).$$
\end{lemma} \begin{proof} If  $|\lambda| > 2 p(T) \epsilon$ then  $$\sum_{r=0}^{\infty} \, \Vert \lambda^{-r-1} \,  (\overline{\Delta}_U T_n)^r
\Vert \leq |\lambda|^{-1} \,  \sum_{r=0}^{\infty} \,  \sum_{s=0}^{k-1} \, {r \choose s} \, (\frac{2 p(T)}{|\lambda|})^r \, \epsilon^{r-s} ,$$
by Lemma \ref{lem5}, for $n \geq \max \{ k, m \}$, where $k, m$ are as in that lemma.   However the latter quantity equals
$$|\lambda|^{-1} \,  \sum_{s=0}^{k-1} \, \sum_{r=0}^{\infty} \,   {r \choose s} \, (\frac{2 p(T) \epsilon}{|\lambda|})^{r-s} \,
(\frac{2 p(T)}{|\lambda|})^{s} = |\lambda|^{-1} \,  \sum_{s=0}^{k-1} \,
(\frac{2 p(T)}{|\lambda|})^{s}  \, (1 - \frac{2 p(T) \epsilon}{|\lambda|})^{-s-1}$$
using the binomial formula.   This is finite, so $\sum_{r=0}^{\infty} \,  \lambda^{-r-1} \,  (\overline{\Delta}_U T_n)^r$ converges,
and this  is clearly an inverse for  $\lambda I - \overline{\Delta}_U T_n$.
If $|\lambda| > 4 p(T) \epsilon$, then the sum in the
last displayed equation is dominated by
$$\frac{1}{4 p(T) \epsilon}  \,  \sum_{s=0}^{k-1} \,
(\frac{1}{2  \epsilon})^s \, 2^{s+1} =
\frac{1}{2 p(T) (1 -  \epsilon)} \frac{1 - \epsilon^k}{\epsilon^k}
\leq K(k,p(T),\epsilon) .$$
 We also obtain

\begin{equation} \label{deleq} \Vert (\lambda I - \overline{\Delta}_U T_n)^{-1} \Vert
\leq K(k,p(T), \epsilon). \end{equation}

By increasing $m$ if necessary, we can assume that $c_n - 1 > 2 \, p(T) \, K(k,
p(T), \epsilon)$.   Then by (\ref{lse}) we have
$$\Vert T_n - \overline{\Delta}_U T_n \Vert \leq \frac{p(T)}{c_n - 1} < \frac{1}{2 K(k,
p(T), \epsilon)}.$$
A simple consequence of the Neumann lemma is that if $R$ is invertible and $\Vert H \Vert < \frac{1}{2  \Vert R^{-1} \Vert}$,
then $R + H$ is invertible and $\Vert (R + H)^{-1} \Vert \leq 2 \Vert R^{-1} \Vert$.  Setting $R = \lambda I - \overline{\Delta}_U T_n$ and
$H =  \overline{\Delta}_U T_n - T_n$, we have $$\Vert H \Vert < \frac{1}{2 K(k,
p(T), \epsilon)} < \frac{1}{2  \Vert R^{-1} \Vert}$$  by (\ref{deleq}).   Hence $R + H = \lambda I - T_n$ is
invertible, and by (\ref{deleq}) again the norm of its inverse is dominated by $2 \Vert R^{-1} \Vert \leq 2
K(k, p(T), \epsilon).$
 \end{proof}

The quantity $K(k,p(T), \epsilon)$ above is independent of $n$, which gives:

\begin{corollary} \label{isnotma}     The spectrum of every element of $A$  is finite
or a null sequence and zero.   Hence $A$ is a modular annihilator algebra.
\end{corollary}
\begin{proof}   Let $0 \neq T = (T_n) \in A$.  We will show that the spectrum of $T$ is finite
or a null sequence and zero.   It is sufficient to show that if $\epsilon > 0$ is given, there exists $m_0 \in \Ndb$
such that if $|\lambda| > 4 p(T) \epsilon$, and if $\lambda$ is not in the spectrum of $T_1, \cdots , T_{m_0}$,
then $\lambda \notin {\rm Sp}_A(T)$.   So assume these conditions, and let $m_0 = \max \{ k, m \}$ as in
Lemma \ref{lem6}.  For $n \geq m_0$ we have by Lemma \ref{lem6} that $\lambda I - T_n$ is invertible, and the usual matrix norm of its inverse
is bounded independently of $n$.
 By assumption this is also true for $n < m_0$.   By Lemma \ref{lem4} there is a $q$ such that
$\Vert d_n T_n d_n^{-1} \Vert < \epsilon$ for $n \geq q$.  If $|\lambda| > \epsilon$ then $(\lambda I - T_n)^{-1} =
\sum_{r=0}^{\infty} \,  \lambda^{-r-1} \,  T_n^r$ and
$$\Vert d_n (\lambda I - T_n)^{-1} d_n^{-1} \Vert =
\Vert \sum_{r=0}^{\infty} \,  \lambda^{-r-1} \,  d_n T_n^r d_n^{-1} \Vert \leq \sum_{r=0}^{\infty} \, |\lambda|^{-r-1} \, \epsilon^r
= |\lambda|^{-1} \, (1 - \frac{\epsilon}{|\lambda|})^{-1} .$$
Thus $(p_n((\lambda I - T_n)^{-1}))$ is bounded independently of $n$.   Hence  $((\lambda I - T_n)^{-1}) \in N$,
and this is an inverse in $N$ for $\lambda I - T$.  Thus the spectrum of $T$ in $N$ is finite 
or a null sequence and zero.    The  spectrum in $A$ might be bigger, but since the boundary of its spectrum
cannot increase, Sp$_A(T)$ is also finite
or a null sequence and zero.

The last statements follow from \cite[Chapter 8]{Pal}.  \end{proof}

We point out some more  features of our example $A$, in hope that these may further its future use as a counterexample in
the subject.

We recall that 
the multiplier algebra $M(A)$ of $A$ is identified with the
idealizer of $A$ in its bidual $A^{**}$ (that is, the set of elements
$\alpha\in A^{**}$ such that $\alpha A\subset A$ and $A
\alpha\subset A$).
 It can also be viewed as the
idealizer of $A$ in $B(H)$, if $A$ is represented nondegenerately and 
completely isometrically on a Hilbert space $H$.  See \cite[Section 
2.6]{BLM} for this.

\begin{proposition} \label{multis}  The multiplier algebra of $A$ may be taken to be
$\{ x \in N : x A + Ax \subset A \}$.  This is also valid with $N$ replaced by $M$.
\end{proposition}

\begin{proof}
Viewing $M = \oplus_n^\infty \, (M_n  \oplus
M_n)$ as represented on $H =\oplus^2_n (\Cdb^n \oplus \Cdb^n)$, it is clear that $D_0$, and hence also $A$,
acts nondegenerately on $H$.  So the multiplier algebra $M(A)$ may be viewed as a subalgebra
of $B(H)$.
We also see that the weak* continuous extension $\tilde{\pi} : A^{**} \to N$ of the `identity map' on $A$,
is a completely isometric homomorphism from the copy of $M(A)$ in $A^{**}$  onto   the copy of $M(A)$ in $B(H)$,
and in particular, the latter is contained in $N$.    So the latter is
$M(A) = \{ x \in N : x A + Ax \subset A \}$.    A similar argument works with $M$ replaced by $N$.
\end{proof}

We note that if $D_n$ is the commutative diagonal $C^*$-algebra
in $M_n$, then there is a natural isometric copy $D$  of $\oplus_n^\infty \, D_n$ inside $N$, namely
the tuples $((x_n , x_n))$ for a bounded sequence  $x_n \in D_n$.

We assume henceforth that $c_n > 1$ for all $n$.

In the next results $\Delta(A)$ denotes the `diagonal' $A \cap A^*$ of $A$
(here $A^*$ is the set of `adjoint operators' (or `involutions') of elements in $A$).  See 2.1.2 in \cite{BLM}.

\begin{proposition} \label{diaga}  The diagonal $\Delta(A)$ equals
 the natural copy $D_0$ of   the
$c_0$-sum $C^*$-algebra $\oplus_n^\circ \, D_n$ inside $A$.
\end{proposition}
\begin{proof}       If $((x_n , d_n x_n d_n^{-1}))$ is selfadjoint,
then $x_n$ is  selfadjoint, and $d_n x_n d_n^{-1}$ is selfadjoint, which forces $d_n^2$ to commute with $x_n$.
However this implies that $x_n$ is diagonal.  Since $\Delta(N) = N \cap N^*$ is spanned by its selfadjoint
elements it follows that $\Delta(N) = D$.
Therefore  $\Delta(A) = D \cap A$,  and this contains $D_0$ since $D_0 \subset A_0 \subset A$
by Corollary \ref{isab}.  The reverse containment follows easily from Lemma \ref{lem4}, but we give
a shorter proof. Let  $(a_n) \in  D \cap A$, with $a_n \in D_n$
for each $n$.  If $\epsilon > 0$ is given, choose $k$ such that $p(u_k (a_n)  - (a_n)) < \epsilon$.
Choose $m$ with  $u_{n,k}$ strictly upper triangular for all $n \geq m$.  Then for
$n \geq m$ we have $|a_n(i) |$, which is the modulus of the  $i$-$i$ entry of $(u_k (a_n)  - (a_n))$, is
dominated by $$\Vert  u_{n,k} \, a_n - a_n \Vert \leq p(u_k  \, (a_n)  - (a_n)) < \epsilon.$$
Thus $\Vert a_n \Vert < \epsilon$ for $n \geq m$, so that $(a_n) \in D_0$.     \end{proof}

We recall some notation from e.g.\ \cite[Chapter 2]{BLM} and \cite{BHN}.
By a {\em projection} we mean an orthogonal projection.  The
second dual $A^{**}$ is also an operator algebra with its (unique)
Arens product, this is also the product inherited from the von Neumann
algebra $B^{**}$ if
$A$ is a subalgebra of a $C^*$-algebra $B$.
Note that $A^{**}$ has an identity $1_{A^{**}}$ of norm $1$ since 
$A$ has a cai.
We say that a projection $p \in A^{**}$ is an {\em open projection}
if there is a net $x_t \in A$ with $x_t = p x_t  \to p$ weak*,
or equivalently with $x_t = p x_t p \to p$ weak* (see  \cite[Theorem 2.4]{BHN}).  These are
also the open projections $p$ in the sense of Akemann \cite{Ake2} in $B^{**}$, where $B$ is a $C^*$-algebra containing $A$, such that
$p \in A^{\perp \perp}$.   The complement $p^\perp = 1_{A^{**}} - p$
 of an open projection for $A$
 is called a {\em closed projection} for $A$.

\begin{corollary} \label{ispro}    Projections in  $A^{**}$ which are both open and closed,
or equivalently (by  \cite[Example 2.1]{BHN} and the first lines of the proof of
\cite[Proposition 2.12]{ABS}) which are in the multiplier algebra
$M(A)$, 
must be also in $D$.  Thus they are diagonal matrices
with  $1$'s as the only permissible nonzero entries.
\end{corollary}
\begin{proof}   This follows from Proposition \ref{multis} and the fact from the proof of
Proposition \ref{diaga} that $\Delta(N) = D$.
\end{proof}

{\bf Remark.}  Note that the natural approximate identity for $\Delta(A) = D_0$ is not an approximate identity for $A$
(since $D_0 A \subset A_0 A \subset A_0 \neq A$).   Thus $A$ is not $\Delta$-dual in the sense of \cite{ABS}.    By \cite{Read} we know that
A has an approximate identity which is `positive' in a certain sense.

\bigskip

We recall that an {\em r-ideal} in $A$ is a right ideal with a left cai, and an {\em $\ell$-ideal} is a left ideal with a right cai.
These objects are in bijective correspondence with the
open projections in $A^{**}$.  Indeed, the limit of such 
one-sided cai in $A^{**}$ exists, and is an open projection
in $A^{**}$ called the {\em support projection} of the 
one-sided ideal.   Conversely, if $p$ is an open projection 
in $A^{**}$ then $\{ a \in A : pa = a \}$ is an r-ideal
(and replacing $pa$ here by $ap$ gives an $\ell$-ideal).  

We recall from \cite{ABS} that $A$ is  {\em nc-discrete} if all the open projections in $A^{**}$ are
also closed (or equivalently, as we said above,  lie in the multiplier algebra $M(A)$).
In \cite[p.\ 76]{ABR} we asked if every approximately unital (semisimple) operator algebra which is a modular annihilator algebra, is weakly compact, or is nc-discrete in the sense of \cite{ABS}.     
In \cite{ABR} we showed that any operator algebra which  is weakly compact is nc-discrete.
 To see that our example  $A$ is not nc-discrete 
note that $A_0$ is an r-ideal in $A$
(and an $\ell$-ideal), and its support projection $p$  in $A^{**}$, which is central in $A^{**}$, coincides with the support projection of $D_0$ in $A^{**}$,
and this is an open projection in  $A^{**}$ which we will show is not closed.

\begin{corollary} \label{iscr}    The algebra $A$ above is not nc-discrete.
\end{corollary}
\begin{proof}  We saw that $p$ above was open.   If $p$ also was closed in $A^{**}$, or equivalently in the multiplier 
algebra $M(A)$, then
$\tilde{\pi}(1-p)$ would be a nonzero central projection in the  copy of $M(A)$ in $M$.
Also $\tilde{\pi}(1-p) e^n_{i,i}$ is nonzero for some $n$ and $i$, because the strong operator topology sum of the $e^n_{i,i}$ in $M$ is $1$.
On the other hand, since $e^n_{i,i}$ is in the ideal supported by $p$ we have $$\tilde{\pi}(p) e^n_{i,i}  =
 \tilde{\pi}(p \, e^n_{i,i}) = \tilde{\pi}(e^n_{i,i}) =e^n_{i,i},$$ and so
$$\tilde{\pi}(1-p) \,  e^n_{i,i} = \tilde{\pi}(1-p) \,  \tilde{\pi}(p) \, e^n_{i,i}  = 0 .$$  This contradiction shows that
$A$ is not nc-discrete.    \end{proof}

Indeed $A_0$ is a nice r- and $\ell$-ideal in $A$ which is supported by an  open projection which
is not one of the obvious projections, and is not any projection in $M(A)$.
Note that $A$
is not
a left or right annihilator algebra in the sense of e.g.\ \cite[Chapter 8]{Pal},
since for example by \cite[Chapter 8]{Pal} this implies that $A$ is compact, whereas
above we showed that $A$ is not even weakly compact.    The spectrum of $A$
is discrete, and every left ideal of $A$ contains a minimal
left ideal, by \cite[Theorem 8.4.5 (h)]{Pal}.  Also every idempotent in $A$ belongs
to the socle by \cite[Theorem 8.6.6]{Pal}, hence to
$A_{00}$ by the next result.   From this it is clear what all the
idempotents  in $A$ are.

\begin{corollary} \label{issoc} The maximal modular right (resp.\ left) ideals in
$A$ are exactly the ideals of the form $(1-e) A$ (resp.\ $A(1-e)$) for
a minimal idempotent $e$ in $A$ which is the canonical copy in $A$ of
a minimal idempotent in $M_n$ for some $n \in \Ndb$.  The socle  of $A$
is $A_{00}$, namely the set of $(a_n) \in A$ with $a_n = 0$ except for at most finitely
many $n$.
\end{corollary}
\begin{proof}   Let $e = (e_n)$ be a (nonzero) minimal idempotent in $A$.
Then $e_n$ is an idempotent in $M_n$ for each $n$.
If $e^n_{i,i}$ is as above, then because the strong operator topology sum of the $e^n_{i,i}$ in $M$ is $1$,
we must have $e e^n_{i,i} e \neq 0$ for some $n$ and $i$.
Since $e$ is minimal, for such  $n$, $e$ is in the copy of $M_n$ in $A_0$.
So this $n$ is unique, and $e$ is clearly a
minimal idempotent in this copy of $M_n$ in $A_0$.  Now it is
easy to see the assertion about the socle of $A$.   By \cite[Proposition 8.4.3]{Pal},
it follows that the maximal modular left ideals  in $A$  are the  ideals
$A(1-e)$ for an $e$ as above.
We have also used the fact here that $A$ has no right annihilators in $A$.
Similarly for right ideals.
 \end{proof}

\begin{corollary} \label{oncom}    The only compact projections 
(in the sense of {\rm \cite{BNII}}) in $A^{**}$ for
the algebra $A$ above are the obvious `main diagonal' ones; that is the projections in
$D_0 \cap A_{00}$.
\end{corollary}
\begin{proof}  Let $T = (T_n) \in A$, and $\epsilon \in (0,\frac{1}{4 p(T)})$ be given.  As in the proof
of Corollary \ref{isnotma} there exists $m_0 \in \Ndb$
such that if $|\lambda| > 4 p(T) \epsilon$
then
 $\lambda I - T_n$ is invertible for $n \geq m_0$, and the usual matrix norm of its inverse
is bounded independently of $n  \geq m_0$.   As in that proof, if $S_n = T_n$ for $n \geq m_0$, and
$S_n = 0$ for $n < m_0$, then $\lambda I - S$ is invertible in $N$.  Thus the spectral
radius $r(S) \leq 4 p(T) \epsilon < 1$. Hence $\lim_{k \to \infty} \,  S^k = 0$ in norm.
Let $q$ be the central projection in $A$ corresponding
to the identity of $\oplus_{n=1}^{m_0 -1} \, M_n$.
If now also $T \in \frac{1}{2} {\mathfrak F}_A$, then $T^k$ converges
weak* to its peak projection $u(T)$ weak*
by \cite[Lemma  3.1, Corollary 3.3]{BNII}, as $k \to \infty$.
Thus $T^k q \to u(T) q$ and $T^k (1-q) = S^k  \to u(T) (1-q)$ weak*.
Clearly it follows that $u(T) q$ is a projection in $A$, hence in $D_0 \cap A_{00}$ as we
said above. On the other hand, since $S^k  \to 0$ we have $u(T) (1-q) = 0$.
Thus $u(T)$ is a projection in $D_0 \cap A_{00}$.

Finally we recall from \cite{BNII} that the compact projections in $A^{**}$
are decreasing limits of such $u(T)$.    Thus any
compact projection is in $D_0 \cap A_{00}$. \end{proof}

One may ask if there exists a  {\em commutative} semisimple approximately unital operator algebra which is a modular annihilator algebra but is not weakly compact.   Later,
after this paper was submitted
we were able to check that the algebra
constructed in  \cite{BRIV} was such an algebra.  However this example
 is quite a bit more complicated than the interesting noncommutative example above.

\section{A complementary example}
 
In \cite[p.\ 76]{ABR} we asked if for an approximately unital commutative operator algebra $A$, which is an ideal in its bidual (or equivalently
 that multiplication by any fixed element of $A$ is weakly compact), is the spectrum of every element at most countable; and is the spectrum of $A$ scattered?   In particular, is it a
modular annihilator algebra  (we recall that
compact semisimple algebras are modular
annihilator algebras \cite[Chapter 8]{Pal}).
There is in fact an easy semisimple counterexample to these questions, which is quite well known in other contexts.
The algebra $A$ will in fact be unital and  isomorphic to a Hilbert space, so is  Banach space reflexive, hence
is obviously  an ideal in its bidual.  It is also singly generated by an operator $T$, and the identity $I$, so that by basic Banach algebra theory the spectrum 
of $A$ is homeomorphic to Sp$_A(T)$.
The example may be described either in the operator theory language of weighted unilateral shifts, and the $H^p(\beta)$ spaces that occur there, or in the Banach algebra language of weighted convolution algebras $l^p(\Ndb_0,\beta)$.
These are equivalent (in particular, $H^2(\beta) = l^2(\Ndb_0,\beta)$). 
We  begin with the Banach algebraic angle:
The weighted convolution algebras $l^1(\Ndb_0,\beta)$ are much studied (see e.g.\ \cite{Dal}), and 
they are Banach algebras
whenever the weight $\beta$ is an ``algebra weight", i.e. 
$\beta_{i+j}\le \beta_i \beta_j$ for all $i,j$. Sometimes, moreover,
the weighted $l^2$ space $l^2(\Ndb_0,\beta)$ is a Banach algebra under the 
convolution product, and in such
cases it is an operator algebra that is isomorphic (as Banach space) to a 
Hilbert space. One such case is
the weight $\beta_n=C(1+n)$ for suitable $C>1$. In any such case the generator 
acts on $l^2(\Ndb_0,\beta)$ as a
weighted shift operator which is unitarily equivalent to a weighted shift on 
$l^2(\Ndb_0)$ with weights $w_i=\beta_{i+1}/\beta_i$. 

From the
 operator theory angle, 
 in the 1960's and 70's, operator theorists exhaustively studied weighted shifts and the algebras
they generate.  See e.g.\ Shields' 1974 survey  \cite{Sh} for this 
and the details below.
Let $T$ be a weighted unilateral shift which is  one-to-one (that is,
none of the weights $w_n$ are zero), and let $A$ be  the algebra  generated by $T$.
Then $A$ is  isomorphic to a Hilbert space if  $T$ is  {\em strictly cyclic} in Lambert's sense \cite{Sh},
that is there is a vector $\xi \in H$ such that $\{ a \xi : a \in A \} = H$.  
Central to the theory of weighted shifts is the convolution algebra
$H^2(\beta) =  l^2(\Ndb_0,\beta)$, and its space of `multipliers' 
$H^\infty(\beta)$.  These spaces can canonically be viewed as spaces of converging (hence analytic) power series on a disk, via the map $(\alpha_n) \mapsto \sum_{n = 0}^\infty \, \alpha_n z^n$.  Here $\beta$ is a sequence related to the weights $w_n$ above by the formula
$\beta_n = w_0 w_1 ... w_{n-1}$.  For example,
one such sequence is given by $\beta_n = n+1$, an example mentioned in the
last paragraph, and the spectral radius of the weighted shift here is 1.
By the well known theory in \cite{Sh}, $T$ is
unitarily equivalent to multiplication by $z$ on $H^2(\beta)$, the latter viewed as a space of power series on the disk
of radius $r(T)$.   In our strictly cyclic case, $A$, which equals its weak closure, is unitarily equivalent via the same unitary to
 $H^\infty(\beta)$.  Since $H^\infty(\beta) H^2(\beta) \subset H^2(\beta)$ and the constant polynomial is in $H^2(\beta)$, it is clear that $H^\infty(\beta) \subset H^2(\beta)$.  However, since the constant polynomial $1$ is a strictly cyclic vector, we in fact have $H^\infty(\beta) = H^\infty(\beta) 1 = H^2(\beta)$  (see p.\ 94 in \cite{Sh}).
On the same page of that reference we see that the closed disk $D$ of radius $r(T)$ is the maximal ideal space of $H^\infty(\beta)$, and the spectrum of any $f \in H^\infty(\beta)$ is $f(D)$.  In particular, Sp$_A(T) =D$, and $A$ is semisimple.  We remark in connection with a 
discussion with Dales, that this implies that $A$ is a natural Banach function
algebra on the disk, but it is not a Banach sequence space in the sense of
Section 4.1 in \cite{Dal}, since $A$  contains
no nontrivial idempotents (since $A$ may be viewed as analytic functions on a disk).

\medskip

{\em Acknowledgements and Remarks.}   This paper was originally the last 
part of a preprint entitled ``Operator algebras with contractive approximate identities III''.
The first author wishes to thank the departments
 of Mathematics at the Universities of
Leeds and Lancaster, and in particular the second author and Garth Dales,
for their warm hospitality during a visit in April--May 2013.     We also gratefully acknowledge   support from UK research council
grant  EP/K019546/1
for largely funding that visit.

\end{document}